\documentclass{amsart}

\usepackage{amsmath}
\usepackage{amssymb}
\usepackage[active]{srcltx}
\usepackage{graphicx,color}%psfrag

\newtheorem{thm}{Theorem}[section]
\newtheorem{prop}[thm]{Proposition}
\newtheorem{lem}[thm]{Lemma}

\theoremstyle{definition}

\newtheorem{remark}[thm]{Remark}

\newcommand{\mat}[2]{\ensuremath{\left[\begin{array}{#1}#2\end{array} \right]}}

\newcommand{\ands}{\quad\mbox{and}\quad}

\newcommand{\sH}{{\mathcal H}}

\newcommand{\sE}{{\mathcal E}}

\newcommand{\sX}{{\mathcal X}}

\newcommand{\im }{{\mbox{ran\,}}}
\def\de{\Delta}

\def\g{\gamma}
\def\ga{\Gamma}

\def\l{\lambda}
\def\la{\Lambda}

\def\t{\tau}

\def\o{\omega}
\def\om{\Omega}

\def\tht{\Theta}

\def\z{\zeta}
\def\ts{\times}

\def\iy{\infty}
\def\im{{\rm Im\, }}

\def\kr{{\rm Ker\, }}

\def\BC{{\mathbb C}}
\def\BD{{\mathbb D}}

\def\BP{{\mathbb P}}

\def\BT{{\mathbb T}}

\begin{document}
\title[Maximum entropy solution to a suboptimal rational Leech problem]
{State space formulas for a suboptimal rational Leech problem I: Maximum entropy solution}

%---------Author 1
\author[A.E. Frazho]{A.E. Frazho}

\address{%
Department of Aeronautics and Astronautics, Purdue University\\
West Lafayette, IN 47907, USA}

\email{frazho@ecn.purdue.edu}

%---------Author 2

\author[S. ter Horst]{S. ter Horst}

\address{%
Unit for BMI, North-West University\\
Private Bag X6001-209, Potchefstroom 2520, South Africa}

\email{sanne.terhorst@nwu.ac.za}

%---------Author 3

\author[M.A. Kaashoek]{M.A. Kaashoek}

\address{%
Department of Mathematics,
VU University Amsterdam\\
De Boelelaan 1081a, 1081 HV Amsterdam, The Netherlands}

\email{m.a.kaashoek@vu.nl}

%------------classification, key words, date



\date{}

\begin{abstract}
For  the strictly positive case (the suboptimal case) the maximum entropy solution $X$ to the Leech problem $G(z)X(z)=K(z)$ and $\|X\|_\iy=\sup_{|z|\leq 1}\|X(z)\|\leq 1$, with $G$ and $K$ stable rational matrix functions, is proved to be a stable rational matrix function. An explicit state space realization for $X$ is given, and $\|X\|_\iy$ turns out to be  strictly less than one. The matrices involved in this realization are computed from the matrices appearing in a state space realization of the data functions $G$ and $K$.  A formula for the entropy of $X$ is also given.
\end{abstract}

\maketitle

\setcounter{section}{0}
\setcounter{equation}{0}

\section{Introduction}\label{intro}

Let  $G$  and $K$ be  matrix-valued $H^\iy$ functions on the open unit disc $\BD$ of sizes $m\ts p$ and $m\ts q$, respectively, and let  $T_G$ and $T_K$  denote  the  corresponding block lower triangular  Toeplitz operators,
\[
T_G:\ell^2_+(\BC^p)\to\ell^2_+(\BC^m), \quad T_K:\ell^2_+(\BC^q)\to\ell^2_+(\BC^m).
\]
A $p\ts q$ matrix-valued $H^\iy$ function $X$ is called a \emph{solution to the Leech problem associated with $G$ and $K$} whenever
\begin{equation}\label{Leech1}
G(z)X(z)= K(z)    \quad (z\in \BD)  \ands \|X\|_\iy=\sup_{z\in\BD}\|X(z)\|\leq 1.
\end{equation}
The Leech problem is an example of a metric constrained interpolation problem, the first part of \eqref{Leech1} is the interpolation condition, and the second part is the metric constraint.   In a note dating from 1971/1972, only published recently \cite{Lch}, see also \cite{LchCom}, Leech proved that  the problem is solvable if and only if the operator $T_GT_G^*-T_KT_K^*$ is nonnegative. Later  the Leech theorem was  derived as a corollary of more general results; see, e.g.,  \cite[page 107]{RR85},  \cite[Section VIII.6]{FF90}), and \cite[Section 4.7]{BW11}.

Now assume in addition that $G$ and $K$ are rational. In other words, assume that $G$ and $K$ are  stable rational matrix functions. In that case, if the Leech problem associated with $G$ and $K$  is solvable, one expects the problem to   have  a stable rational matrix   solution as well.  However, a priori this is not clear, and the existence of rational solutions  was  proved only recently in  \cite{Trent13} by reducing the problem to polynomials, in  \cite{tH13} by adapting the lurking isometry method used in \cite{BT98}, and in \cite{FtHK13} by using a state space approach.

In  the  present paper    $G$ and $K$ are also stable rational matrix functions. We assume additionally that  the operator $T_GT_G^*-T_KT_K^* $ is   strictly positive. It is then known from commutant lifting theory    that the  Leech problem  has  a unique maximum entropy solution, that is, the (unique) solution $X$ to the Leech problem associated with  $G$ and $K$ for which the quantity
\begin{equation}\label{ent0}
\mathcal{E}(X) = \frac{1}{2\pi} \int_0^{2\pi}
\ln \det[ I_q - X(e^{\imath \omega})^*X(e^{\imath \omega})] d \omega
\end{equation}
is maximal.  In this paper we show that this  maximum entropy solution is   a stable rational   matrix function,  we derive an explicit formula for this  solution  and a formula for its entropy $\mathcal{E}(X)$; see Theorem \ref{mainthm} below.  When $T_GT_G^*-T_KT_K^* $ is only non-negative, the maximum entropy solution still exists  but the problem whether or not it is rational remains open.

To prove the above mentioned results,  we use  the fact,  well-known from mathematical systems theory  (see, e.g., Chapter 1 of \cite{CF03} or Chapter 4 in \cite{BGKR08}),    that rational matrix functions admit state space realizations. For  our $G$ and $K$   this means that  the matrix function $\begin{bmatrix}G & K \end{bmatrix}$ admits a representation of the following form:
\begin{equation}
\label{realGK}
\begin{bmatrix}
  G(z)  & K(z)
\end{bmatrix}
 = \begin{bmatrix}
  D_1   & D_2
\end{bmatrix}  + z C(I_n -  z A)^{-1}\begin{bmatrix}
  B_1   & B_2
\end{bmatrix} .
\end{equation}
Here $I_n$ is the   $n\ts n$ identity matrix,   $A$ is  an $n\ts n$ matrix, and $B_1$, $B_2$, $C$, $D_1$ and $D_2$ are matrices of appropriate sizes. Moreover, since $G$ and $K$ are stable rational matrix  functions,  $G$ and $K$ have no pole in the closed unit disc, and therefore we may assume that  matrix $A$ is \emph{stable}, that is, $A$ has all its eigenvalues in the open unit disc. The realization \eqref{realGK} is called \emph{minimal} if  there exists no realization of $\begin{bmatrix}G & K \end{bmatrix}$ as in \eqref{realGK} with `state matrix' $A$   of smaller size than the one in the given realization. In that case the order $n$ of $A$ is called the \emph{McMillan degree} of  $\begin{bmatrix}G & K \end{bmatrix}$. If  the realization \eqref{realGK} is minimal, then the matrix $A$ is automatically stable and the observability operator $W_{obs}$, which is defined by
\begin{equation}\label{defWobs}
W_{obs}=\begin{bmatrix}
C  \\
C A  \\
C A^2 \\
\vdots
\end{bmatrix}:\BC^n\to \ell_+^2(\BC^m),
\end{equation}
is one-to-one. In the sequel we do not require the realization \eqref{realGK} to be minimal but  we shall always assume that $A$ is stable and $W_{obs}$ is one-to-one.  In that case we refer to \eqref{realGK} as an \emph{observable stable realization}.

As a first step towards our main result we first derive, in Theorem \ref{thmpos} below, a necessary and sufficient condition for  $T_GT_G^*-T_KT_K^* $ to be strictly positive in terms of the matrices in \eqref{realGK} and related matrices. To do this we need  the rational ${m\ts m}$ matrix function
\begin{equation}
\label{RGK}
R(z)=G(z)G^*(z)  -K(z)K^*(z) .
\end{equation}
Here $G^*(z) =G(\bar{z}^{-1})^*$  and $K^*(z) =K(\bar{z}^{-1})^*$.  Note that $R$ has no pole on the unit circle $\BT$.  By $T_R$ we denote the Toeplitz operator defined by $R$. Using the realization \eqref{realGK} one shows (see \cite[Lemma 3.1]{FtHK13}) that  $R$  admits the following state space representation:
\begin{equation}\label{eqR}
R(z) = zC(I - zA)^{-1}\ga + R_0 + \ga^*(zI - A^*)^{-1}C^*.
\end{equation}
Here  $R_0$ and $\ga$ are matrices of sizes  $m\ts m$ and $n\ts m$, respectively, defined by
\begin{align}
R_0 &= D_1D_1^* -D_2D_2^* +C(P_1 - P_2)C^*,\label{defR0}\\[.1cm]
 \ga &= B_1D_1^* -B_2D_2^* + A(P_1 - P_2)C^*,  \label{defGa}
  \end{align}
and $P_1$ and $P_2$ are the unique $n\ts n$ matrix solutions of   the symmetric  Stein equations:
\begin{equation}\label{p1p2}
P_1-AP_1A^*=B_1B_1^*\quad\mbox{and}\quad P_2-AP_2A^*=B_2B_2^*.
\end{equation}
Since $A$ is stable, the above equations are solvable and the solutions are unique. Finally, we associate   with $R$  the  algebraic Riccati equation:
\begin{equation}\label{are}
Q  =  A^* Q  A + ( C - \Gamma^* Q  A )^*
( R_0 - \Gamma^* Q  \Gamma )^{-1} ( C - \Gamma^* Q  A ).
\end{equation}
We are now ready to state our main results.

\begin{thm}\label{thmpos}
Let $G$ and $K$ be stable rational matrix functions, and assume that $\begin{bmatrix} G&K\end{bmatrix}$ is given by the observable stable realization \eqref{realGK}. Then the operator $T_G T_G^* - T_K T_K^*$ is strictly positive if and only if
the following two conditions hold.
\begin{itemize}
\item[\textup{(i)}] There exists a strictly positive $n\times n$ matrix $Q$ such that
 \begin{itemize}
\item[\textup{(a)}] $R_0 - \Gamma^* Q  \Gamma$  is strictly positive,
\item[\textup{(b)}] $Q $ satisfies the Riccati equation \eqref{are},
\item[\textup{(c)}] the matrix
 $A_{0} = A-\ga( R_0 - \Gamma^* Q  \Gamma )^{-1} ( C - \Gamma^* Q  A )$ is stable.
\end{itemize}
\vspace{.1cm}
\item[\textup{(ii)}] The   operator
$Q^{-1} + P_2 - P_1$ is strictly positive.
\end{itemize}
In this case, the Toeplitz operator $T_R$ is strictly positive and  the inverse of the operator $T_G T_G^* - T_K T_K^*$   is given by
\begin{align}
&\hspace{1cm}\Big(T_GT_G^*-T_KT_K^*\Big)^{-1} =
T_R^{-1} + T_R^{-1}W_{obs}\om W_{obs}^* T_R^{-1}, \ \mbox{where}
\nonumber \\
&\om = (P_1-P_2)(Q^{-1}  + P_2-P_1)^{-1}Q^{-1}= (P_1-P_2)\big(I+Q(P_2-P_1)\big)^{-1}.  \label{defOm}
\end{align}
\end{thm}

The second main result shows that the maximum entropy solution is rational and provides a state space realization for this solution.

\begin{thm}\label{mainthm}
Let $G$ and $K$ be stable rational matrix functions, and assume that $\begin{bmatrix} G&K\end{bmatrix}$ is given by the observable stable realization \eqref{realGK}. Furthermore, assume that $T_G T_G^* - T_K T_K^*$ is strictly positive, or equivalently, that items $(i)$ and $(ii)$ of  \textup{Theorem \ref{thmpos}} hold. Then the maximal entropy solution $X$   to the Leech problem for $G$ and $K$  is a stable rational matrix function which is given by the following state space realization:
\begin{equation}\label{realx0}
X(z) = D_U D_V^{-1} +z \left( C_1 - D_UD_V^{-1} C_2\right)(I - z A^\times)^{-1} B_0D_V^{-1}.
\end{equation}
Here, using the matrices appearing in Theorem \ref{thmpos},  the  matrices in \eqref{realx0} are defined by
\begin{align}
{\de}& =  R_0 - \Gamma^* Q  \Gamma, \quad C_0= {\de^{-1}} ( C - \Gamma^* Q  A ), \quad A_0=A-\ga C_0; \nonumber\\
C_{j}& =  D_j^* C_0   + B_j^* Q A_0, \quad (j=1,2);\label{defC12}\\
{D_0} & = {\de^{-1}}(D_2 - \Gamma^*Q B_2) + C_0\om C_2^*;\label{defD0}\\
B_0& = B_2   - \ga\de^{-1}(D_2 - \ga^*Q B_2) +  A_0 \om C_2^*;\label{defB0}\\
D_{U} & =  D_1^* {D_0} + B_1^* Q B_0, \quad D_{V}  = I_q + D_2^* {D_0} + B_2^* Q B_0; \label{defDUDV}\\
A^\times & = A_0 - B_0 D_V^{-1} C_2.\nonumber
\end{align}
Moreover, the state   matrix $A^\times$ is stable, the matrix $D_V$ is strictly positive, and the entropy of  $X$ is given by
\begin{equation}\label{entx}
 \mathcal{E}(X) = -\ln \det[D_{V}].
\end{equation}
Finally, $\|X\|_\iy=\sup_{|z|\leq 1}\|X(z)\|$ is strictly less than one, and the McMillan degree of $X$ is less than or equal to the McMillan degree of $\begin{bmatrix}  G & K   \end{bmatrix}$.
\end{thm}

A description  of all solutions to the Leech problem \eqref{Leech1} for the case when $G$ and $K$ are rational will be the topic of a future publication.

\begin{remark}\label{remQ}
An $n\times n$ matrix $Q$ is said to be  a \emph{stabilizing solution} of  the algebraic Riccati equation \eqref{are} whenever  $Q$  satisfies the three conditions (a), (b) and (c) in item (i) of Theorem \ref{thmpos}. In this case  $Q$ is uniquely determined, cf., formula \eqref{QTRinv} below. Moreover, the existence  of a stabilizing solution of \eqref{are} is equivalent to the Toeplitz operator $T_R$ being strictly positive. In that case, the stabilizing matrix $Q$ is given by
\begin{equation}\label{QTRinv}
Q =W_{obs}^* T_R^{-1}W_{obs}.
\end{equation}
See, e.g., \cite[Section 10.3]{FB10}, \cite[Section 14.7]{BGKR2}, and \cite[Theorem 1.1]{FKR1} for a non-symmetric version. Also note that there exist several efficient numerical  algorithms  to compute a stabilizing solution, cf., \cite{AL84}.
\end{remark}

The special case of Leech's theorem with $q=m$ and $K$ identically equal to the $m\ts m$ identity matrix $I_{m}$ is part of the corona theorem, which is due to Carlson \cite{Carl62}, for $m=1$, and Fuhrmann \cite{Fuhr68} for arbitrary $m$. The least squares solution of  the corona version of the equation can be found in \cite{FKR2a} and a description of all solutions without any norm constraint in \cite{FKR2b}. For an engineering perspective on   corona and Leech type problems and related applications in signal processing we refer to \cite{WBP09,WB12} and the references therein.

The paper consists of  five sections including the present introduction. In Section \ref{CLT}  we recall the basic results from commutant lifting theory used in the present paper, and we specify these results for the Leech problem.  In Section \ref{infmodel} we assume that  $T_GT_G^* - T_K T_K^*$ is strictly positive and, using the commutant lifting results, we derive an infinite dimensional state space realization for the maximum entropy solution.  These two sections do not require $G$ and $K$ to be rational; the next two sections do.  In Section \ref{secprthmQ} we further clarify   the role of the Toeplitz operator $T_R$ with $R$ being given by \eqref{RGK} and prove Theorem~\ref{thmpos}. The proof of Theorem~\ref{mainthm} is given in the final section.  {At the end of the final section we present a direct proof   of the fact that the function $X$ given by \eqref{realx0} satisfies the first identity in \eqref{Leech1} (see Remark \ref{rem53}).}

\medskip
\noindent\textbf{Some terminology and notation.} For any positive integer $k$ we write $E_k $ for the canonical embedding of $\BC^k$ onto the first coordinate space of $\ell_+^2(\BC^k)$, that is,
\[
E_k  = \begin{bmatrix}
           I_k & 0 & 0 & 0 & \cdots \,\,\\
         \end{bmatrix}{}^\top:\mathbb{C}^k \rightarrow \ell_+^2(\mathbb{C}^k).
\]
Here $\ell^2_+(\BC^k)$ denotes the Hilbert space of unilateral square summable sequences of vectors in $\BC^k$. By $S_k$ we denote  the unilateral shift on $\ell_+^2(\mathbb{C}^k)$. For positive integers $k$ and $r$ we write $H_{k\ts r}^\iy$ for the Banach space of  all $k\ts r$ matrices with entries from $H^\iy$, the algebra of all bounded analytic functions of the open unit disc $\BD$.  As usual, we identify a $k\ts r$ matrix with complex entries with the linear operator from $\BC^r$ to $\BC^k$ induced  by the action of the matrix on the standard bases. By definition, the infinity norm of   $F\in H_{k\ts r}^\iy$ is given by  $\|F\|_\iy=\sup_{|z|<1} \|F(z)\|$. A function  $F\in H_{k\ts r}^\iy$ is said to be  \emph{outer} if the Toeplitz operator $T_F$ from $\ell_+^2(\BC^r)$ to $\ell_+^2(\BC^k)$  defined by $F$  has a dense range. We call $F\in H_{k\ts k}^\iy$ \emph{invertible outer} if $\det F(z)\not =0$ for each $z\in \BD$ and $F^{-1}$ belongs to $H_{k\ts k}^\iy$. Thus  $F\in H_{k\ts k}^\iy$ is invertible outer if and only if $T_F$ is invertible, and in that case $T_F^{-1}=T_{F^{-1}}$.

\setcounter{equation}{0}
\section{The central commutant lifting solution}\label{CLT}

In this section we recall the construction of the  central solution in the Sz.-Nagy-Foias commutant lifting theorem, as presented in Chapter IV of \cite{FFGK} with  the bound  $\g$ equal to one. Note that in this setting, by \cite[Theorem  IV.7.5]{FFGK}, the central solution is equal to the maximum entropy solution (see Theorem \ref{thc2} below).

\begin{thm}[Commutant lifting] \label{thcl}
Let $\sH^\prime$ be an invariant subspace for the backward shift $S_p^*$ on $\ell_+^2(\mathbb{C}^p)$ and $T^\prime$ the operator on $\sH^\prime$ obtained by compressing $S_p$ to $\sH^\prime$, that is, $T^\prime = P_{\sH^\prime} S_p |\sH^\prime$. Let $\la$ be a contraction mapping $\ell_+^2(\mathbb{C}^q)$  into $\sH^\prime$ satisfying $T^\prime \la = \la S_q$. Then there exists a function $X$ in $H_{p\times q}^\infty$ such that
\begin{equation}\label{clt2}
\la = P_{\sH^\prime} T_X
\quad \mbox{and}\quad \|X\|_\infty \leq 1.
\end{equation}
Moreover, if $\|\la\| < 1$, then  a function $X$ in $H_{p\times q}^\infty$ satisfying $\la = P_{\sH^\prime} T_X$ and $\|X\|_\infty \leq 1$ is given by
\begin{align}\label{clt00}
X(z) &=  U(z)V(z)^{-1}, \nonumber\\
U(z) &= E_p^* \left(I - z S_p^*\right)^{-1} \la\left(I - \la^* \la  \right)^{-1}  E_q, \nonumber \\
V(z) &= E_q^* \left(I - z S_q^*\right)^{-1}  \left(I - \la^* \la  \right)^{-1} E_q.
\end{align}
Moreover, $\det V(z)\not = 0$ for $| z|<1$, the function $V^{-1}$ belongs to  $H_{q\times q}^\infty$ and is  an outer function. In fact,  the function $\tht=V(0)^{1/2} V^{-1}$ is the outer  spectral factor of  the function $I-X^*X$, that is
\begin{equation}
\label{outerfact}
I-X(\z)^*X(\z)=\tht(\z)^* \tht(\z), \quad \z\in \BT \ \mbox{  a.e.}.
\end{equation}
\end{thm}

The formulas for $X$, $U$ and $V$ appearing in \eqref{clt00} and the identity \eqref{outerfact}  are obtained from \cite[Theorem IV.6.6]{FFGK} using $\g=1$ and $A=\la$.

The following theorem (see \cite[Theorem  IV.7.5]{FFGK}) shows that the function $X$ constructed in the second part of the above theorem is the maximum entropy solution.

\begin{thm}\label{thc2}
The function  $X$ in \eqref{clt00} is the maximal entropy solution, that is,
if $Y$ in $H_{p\times q}^\infty$ satisfies $\la = P_{\sH^\prime} T_Y$ and $\|Y\|_\infty \leq 1$,
then  $\sE(Y)\leq \sE(X)$. Moreover, the maximal entropy solution is unique and
\begin{equation}\label{entclt}
\sE(X) = -\ln \det[ V(0)] = -\ln \det[ E_q^*(I - \la^* \la)^{-1}E_q].
 \end{equation}
\end{thm}

\medskip
In the remaining part of this section we will  apply the previous theorems to the special choice of $\la$ associated with our Leech problem.  For this special case $\la$ is given in item (iii) of  the following lemma for the general case when   $G$ and $K$ are  matrix-valued  $H^\infty$ functions and not necessarily rational functions.

\begin{lem}\label{lempos}
Let  $G$  and $K$ be  matrix-valued $H^\iy$  functions of sizes $m\ts p$ and $m\ts q$, respectively, and assume that $T_GT_G^* - T_KT_K^*$ is strictly positive.
 Then the following statements hold.
 \begin{itemize}
   \item[(i)] The operator $T_GT_G^*$ is invertible, or equivalently,
   $T_G^*$ is one-to-one and has closed range.
   \item[(ii)] The subspace $\sH^\prime = \im T_G^*$ is invariant for the
   backward shift $S_p^*$.
   \item[(iii)] The operator $\la = T_G^* \left(T_GT_G^* \right)^{-1}T_K$ viewed as an operator from $\ell_+^2(\BC^q)$ into $\sH^\prime$  is a strict contraction. Moreover,
   \begin{equation}\label{lam3}
    T_G \la = T_K.
   \end{equation}
   \item[(iv)] The operator $\la$ intertwines $S_q$ with $T^\prime$, that is,
   \begin{equation}\label{lam31}
    \quad T^\prime \la = \la S_q
   \end{equation}
 where $T^\prime$ on $\sH^\prime$ is the compression of $S_p$ to $\sH^\prime$,
 that is, $T^\prime = P_{\sH^\prime}S_p |\sH^\prime$.
 \end{itemize}
\end{lem}

\begin{proof}[\bf Proof]
Because  $T_GT_G^* - T_KT_K^*$ is strictly positive, $T_G T_G^*$ is also strictly positive. Hence $T_G^*$ is one-to-one and has closed range. Thus item (i) holds. From item (i) we conclude  that  $\sH'$ is a closed subspace of $\ell_+^2(\mathbb{C}^p)$.  Using $S_p^* T_G^* = T_G^* S_m^*$, it follows that
$\sH^\prime = \im T_G^*$ is an invariant subspace for the
backward shift $S_p^*$. Therefore  item (ii) holds.

Using the definition of $\la$ we see that
\[
T_G\la u=  T_GT_G^* \left(T_GT_G^* \right)^{-1}T_Ku= T_Ku, \quad u\in \ell_+^2(\BC^q).
\]
This proves \eqref{lam3}. We also have $\la^* T_G^* = T_K^* $. It follows that  the operator
\[
T_G(I - \la \la^*)T_G^*  = T_G T_G^* - T_K T_K^*
\]
is strictly positive. Using $\sH^\prime = \im T_G^*$ and the fact that  $T_G^*$ is one-to-one and has closed range, we conclude that
$I - \la \la^*$ is also strictly positive.
In other words, $\la$ is a strict contraction and  item (iii) holds.

Recall   {that } $T^\prime$ is the compression of $S_p$ onto $\sH^\prime$.
Because    $\sH^\prime$ is
an invariant subspace for the backward shift $S_p^*$, we have
 $P_{\sH^\prime} S_p =  T^\prime P_{\sH^\prime}$.
(In the language of the commutant lifting theorem,
$S_p$ is an isometric dilation of $T^\prime$.)
Notice that $\sH^\prime = (\kr T_G)^\perp$. Now observe that
\[
T_G T^\prime \la =  T_G S_p \la    =  S_m T_G  \la =
S_m  T_K = T_K S_q =   T_G \la S_q.
\]
Hence $T^\prime \la = \la S_q$.
Therefore item (iv) holds.
\end{proof}

The above lemma shows that operator   $\la = T_G^* \left(T_G T_G^*\right)^{-1} T_K$ mapping $ \ell_+^2(\BC^q)$ into  $\sH^\prime$ satisfies the hypothesis of Theorem \ref{clt2}. In particular,
$X= U V^{-1}$ in \eqref{clt00} is the maximal entropy solution;
see Theorem \ref{thc2}. In Section \ref{secrealX},  we will construct the
finite dimensional state space realization for $X$ in Theorem \ref{mainthm}.
 According to Theorem \ref{thcl}
the operator $\la = P_{\sH^\prime} T_X$ and $\|X\|_\infty \leq 1$.
Hence $T_K = T_G \la = T_G T_X$, or equivalently, $G X = K$.
Therefore $X$ is a   solution to the Leech problem assocaited to $G$ and $K$.
This $X = U V^{-1}$ is also the unique
maximal entropy solution over the set of all contractive
analytic solutions
$Y$ for  $G Y = K$.
If $Y$ is contractive analytic solution to $G Y =K$, then $T_G T_Y = T_K$ and $\|T_Y\| =\|Y\|_\infty \leq 1$. Notice that
\[
T_G P_{\sH^\prime} T_Y = T_G T_Y = T_K = T_G \la.
\]
Because    $\sH^\prime = (\kr T_G)^\perp$, we have  $\la = P_{\sH^\prime}T_Y$. Theorem \ref{thc2} guarantees that $\mathcal{E}(Y) \leq \mathcal{E}(X)$ with equality if and only if $Y =X$. Since $X = U V^{-1}$,  formula
\eqref{entx} is a direct consequence of Theorem \ref{thc2}.

\setcounter{equation}{0}
\section{The infinite dimensional state space model}\label{infmodel}
Throughout this section  $G\in H_{m\times p}^\infty$  and $K\in H_{m\times q}^\infty$, and   $T_GT_G^* - T_K T_K^*$ is assumed to be strictly positive. Furthermore, $\la = T_G^* \left(T_GT_G^* \right)^{-1}T_K$ is viewed as an operator from $\ell_+^2(\BC^q)$ into $\sH^\prime=\im T_G^*$,  and  $X = U V^{-1}$ is the maximal entropy solution in Theorem \ref{thcl} corresponding to this choice of $\la$.

The following proposition provides an infinite dimensional state space realization for $X$ even when $G$ and $K$ are   nonrational.

\begin{prop} \label{prop_real}
Assume that $T_G T_G^* - T_K T_K^*$ is strictly positive, where  $G$ and $K$ are functions in $H_{m\times p}^\infty$ and $H_{m\times q}^\infty$,  respectively.  Let $\la$ be the strict contraction defined by $\la = T_G^* \left(T_G T_G^*\right)^{-1} T_K$.  Then   the function $X = U V^{-1}$ in \eqref{clt00},   is given by the following infinite dimensional state space realization
\begin{equation}\label{realx}
X(z) = D_U D_V^{-1} +
 z\left(E_p^* T_G^* - D_U D_V^{-1} E_q^* T_K^* \right)  (I - z F)^{-1} S_m^* \Xi  D_V^{-1}
 \quad (z \in \mathbb{D}).
 \end{equation}
Here $\Xi$ and $F$ are the operators defined by
 \begin{align}
  \Xi &=  (T_GT_G^* - T_KT_K^*)^{-1}T_KE_q :\mathbb{C}^q \rightarrow \ell_+^2(\mathbb{C}^m), \label{realXi} \\[.1cm]
 F   &= S_m^* - S_m^* \Xi D_V^{-1} E_q^* T_K^* \mbox{ on }\ell_+^2(\mathbb{C}^m), \label{realF}
 \end{align}
 and $D_U$ and $D_V$  are given by
 \begin{equation} \label{realx1b}
D_U = E_p^* T_G^*  \Xi:\mathbb{C}^q \rightarrow\mathbb{C}^p, \quad
D_V   = I_q + E_q^* T_K^* \Xi:\mathbb{C}^q \rightarrow\mathbb{C}^q.
\end{equation}
Finally, the spectral radius $r_{spec}(F) \leq 1$.
\end{prop}

The following lemma is used to prove the above result.

\begin{lem}\label{leminv0}
Assume that $T_G T_G^* - T_K T_K^*$ is strictly positive,
where  $G$ and $K$ are functions in $H_{m\times p}^\infty$ and
$H_{m\times q}^\infty$,  respectively.
Let $\la$ be the strict contraction defined by
$\la = T_G^* \left(T_G T_G^*\right)^{-1} T_K$. Then
\begin{align} \label{invda2}
(I - \la^* \la)^{-1} &= I + T_K^* (T_GT_G^* - T_KT_K^*)^{-1}T_K,\\
\la(I-\la^*\la)^{-1} &=  T_G^*(T_GT_G^* - T_KT_K^*)^{-1}T_K. \label{UU}
\end{align}
\end{lem}

\begin{proof}[\bf Proof]
Using
$\la =  T_G^*(T_G T_G^*)^{-1} T_K$, we obtain
\[
I - \la^* \la =  I - T_K^* (T_GT_G^*)^{-1}T_K.
\]
The operator  inversion formula  $(I - C^*A^{-1}C)^{-1}= I + C^*(A - CC^*)^{-1}C$ yields the formula for $(I - \la^* \la)^{-1}$ in \eqref{invda2}.

On the other hand,
\begin{align*}
\la(I-\la^*\la)^{-1} &=
T_G^*(T_GT_G^*)^{-1}T_K\Big(I + T_K^* (T_GT_G^* - T_KT_K^*)^{-1}T_K\Big)\\
&=T_G^*(T_GT_G^*)^{-1}\Big(I + T_KT_K^* (T_GT_G^* - T_KT_K^*)^{-1}\Big)T_K\\
&= T_G^*(T_GT_G^*)^{-1}\Big((T_GT_G^* - T_KT_K^*)+ T_KT_K^*\Big) (T_GT_G^* - T_KT_K^*)^{-1}T_K\\
&= T_G^*(T_GT_G^* - T_KT_K^*)^{-1}T_K.
\end{align*}
This yields the formula for $\la(I-\la^*\la)^{-1}$ in \eqref{UU}.
\end{proof}

\begin{proof}[\textbf{Proof of Proposition \ref{prop_real}}]
Recall that $X  = U V ^{-1}$. We first establish a state space realization for $U$.
By employing \eqref{UU}, we obtain
\begin{align*}
U(z) &= E_p^*(I - z S_p^*)^{-1} \la(I - \la^* \la)^{-1} E_q\\
 &= E_p^*(I - z S_p^*)^{-1} T_G^*(T_GT_G^* - T_KT_K^*)^{-1}T_K E_q\\
 &= E_p^* T_G^*(I - z S_m^*)^{-1} (T_GT_G^* - T_KT_K^*)^{-1}T_K E_q\\
 &= E_p^* T_G^*(I - z S_m^*)^{-1} \Xi, \quad |z|<1.
\end{align*}
See \eqref{realXi} for the definition of $\Xi$. Using $D_U = E_p^* T_G^* \Xi$, we see that   a state space realization for $U$ is given by
\begin{equation}\label{realU}
U(z) =  D_U + z E_p^* T_G^*(I - z S_m^*)^{-1}S_m^* \Xi, \quad |z|<1.
\end{equation}

To compute a state space realization for $V$, we use \eqref{invda2} in the following calculation:
\begin{align*}
V(z) &= E_q^*(I - z S_q^*)^{-1}(I - \la^* \la)^{-1} E_q\\
&= E_q^* E_q  + E_q^*(I - z S_q^*)^{-1} T_K^* (T_GT_G^* - T_KT_K^*)^{-1}T_KE_q\\
&= I_q + E_q^* T_K^* (I - z S_m^*)^{-1} \Xi\\
&= I_q + E_q^* T_K^* \Xi + z E_q T_K^* (I - z S_m^*)^{-1} S_m^* \Xi, \quad |z|<1.
\end{align*}
By consulting \eqref{realx1b}, we see that a state space realization for $V$ is given by
\begin{equation}\label{realV}
V(z) =  D_V + z E_q^* T_K^*  (I - z S_m^*)^{-1}  S_m^* \Xi, , \quad |z|<1.
\end{equation}

Using a classical state space inversion formula, the inverse of $V(z)^{-1}$ in a neighborhood of zero is given by
\begin{align*}
V(z)^{-1} &= D_V^{-1} - z D_V^{-1}E_q^* T_K^* (I - z F)^{-1} S_m^* \Xi D_V^{-1},  \mbox{ where}\nonumber\\
       &\hspace{1cm} F = S_m^* - S_m^* \Xi D_V^{-1} E_q^* T_K^*, \mbox{ as in } \eqref{realF}.
\end{align*}
On the other hand, by  the final part of Theorem \ref{thcl}, we know that $V(z)$ is invertible for each $z$ in $\mathbb{D}$. Since $r_{spec}(S_m^*) \leq 1$, we can then apply Theorem 2.1 in \cite{BGKR08} (with $ \l= z^{-1}$) to show that  $r_{spec}(F) \leq 1$.  Thus
\begin{equation}
\label{realVinv}
V(z)^{-1} = D_V^{-1} - z D_V^{-1}E_q^* T_K^* (I - z F)^{-1} S_m^* \Xi D_V^{-1}, \quad |z|<1.
\end{equation}

To compute a state space realization for the maximum entropy  solution $X$, we first observe (using the identity \eqref{realF})  that
\begin{align*}
&- z E_p^* T_G^*(I - z S_m^*)^{-1} S_m^*\Xi D_V^{-1}E_q^* T_K^*  (I - z F)^{-1} S_m^*\Xi D_V^{-1}\\
&\quad = zE_p^* T_G^*(I - z S_m^*)^{-1}\Big(F - S_m^* \Big)(I - z F)^{-1} S_m^*\Xi D_V^{-1}\\
&\quad= E_p^* T_G^*(I - z S_m^*)^{-1}\Big((I  - z S_m^*) - (I - z F) \Big)(I - z F)^{-1} S_m^*\Xi D_V^{-1}\\
&\quad= E_p^* T_G^* (I - z F)^{-1} S_m^*\Xi D_V^{-1} - E_p^* T_G^*(I - z S_m^*)^{-1} S_m^*\Xi D_V^{-1}.
\end{align*}
This readily implies that
\begin{align*}
&- z^2 E_p^* T_G^*(I - z S_m^*)^{-1}S_m^*\Xi D_V^{-1}E_q^* T_K^* (I - z F)^{-1}S_m^* \Xi D_V^{-1}=\\
&\qquad = zE_p^* T_G^* (I - z F)^{-1} S_m^*\Xi D_V^{-1} - zE_p^* T_G^*(I - z S_m^*)^{-1} S_m^*\Xi D_V^{-1}.
\end{align*}
The state space realizations for $U$ in \eqref{realU} and $V^{-1}$ in
\eqref{realVinv} then yield:
\begin{align*}
U(z)V(z)^{-1} &=  D_UD_V^{-1} -z D_UD_V^{-1}E_q^* T_K^* (I - z F)^{-1} S_m^*\Xi D_V^{-1}\\
&\quad + z E_p^* T_G^*(I - z S_m^*)^{-1}S_m^* \Xi D_V^{-1}\\
&\quad -z^2E_p^* T_G^*(I - z S_m^*)^{-1}S_m^*\Xi
 D_V^{-1}E_q^* T_K^* (I - z F)^{-1} S_m^*\Xi D_V^{-1}\\
&= D_U D_V^{-1} +
 z\left(E_p^* T_G^* - D_U D_V^{-1} E_q T_K^* \right)  (I - z F)^{-1} S_m^* \Xi  D_V^{-1}
\end{align*}
when $z \in \mathbb{D}$. This proves the state space formula for $X$ in \eqref{realx}.
\end{proof}

\setcounter{equation}{0}
\section{Proof of Theorem \ref{thmpos}}\label{secprthmQ}
Throughout the section  $G$ and $K$ are stable rational matrix functions of sizes $m\ts p$ and $m\ts q$, respectively, and we assume that $\begin{bmatrix} G&K\end{bmatrix}$ is given by the observable stable realization \eqref{realGK}. We first prove two lemmas.  The first deals with the $m \ts m$ rational matrix function  $R$ defined by \eqref{RGK}.

\begin{lem}\label{lempos1}
Let $R$ be the $m \ts m$ rational matrix function  defined by \eqref{RGK}. Then $T_R$ is strictly positive whenever  $T_G T_G^* - T_K T_K^*$ is strictly positive.
\end{lem}

\begin{proof}[\bf Proof]
Assume that $T_G T_G^* - T_K T_K^*$ is strictly positive. For each $z\in\BD$  put $\varphi_{z,m}=\mat{cccc}{I_m & zI_m & z^2I_m  &    \cdots}^*$. Note that
\[
T_G^* \varphi_{z,m}=\varphi_{z,p}G(z)^*,\quad T_K^* \varphi_{z,m}= \varphi_{z,q}K(z)^*,\quad
\varphi_{z,m}^* \varphi_{z,m}=\frac{1}{1-|z|^2}I_m.
\]
Since  $T_G T_G^* - T_K T_K^*$ is assumed to be strictly positive, there exists an $\eta>0$ such that $T_GT_G^*-T_KT_K^*\geq \eta I$. Multiplying  this inequality by $\varphi_{z,m}$ on the right and by $\varphi_{z,m}^*$ on the left gives
\[
\frac{G(z)G(z)^*-K(z)K(z)^*}{1-|z|^2}\geq \frac{\eta}{1-|z|^2}I_m   \quad (z\in \BD).
\]
Multiplying with $1-|z|^2$ and taking limits $z\to e^{i \omega}$ on the unit circle, shows
\[
R(e^{i\omega})=G(e^{i\omega})G(e^{i\omega})^*-K(e^{i\omega})K(e^{i\omega})^*\geq \eta I_m, \quad 0\leq \o\leq 2\pi.
\]
This implies $T_R\geq \eta I_m$.
\end{proof}

\begin{lem}\label{lempos2} Let $W_{obs}$ be defined by \eqref{defWobs}, and let $P_1$ and $P_2$ be the unique $n\ts n$ matrix solutions of   the  Stein equations \eqref{p1p2}.  Then
\begin{equation}\label{trp1p2}
T_GT_G^*-T_KT_K^* =T_R + W_{obs}(P_2 - P_1)W_{obs}^*.
\end{equation}
In particular,  the operator $T_GT_G^*-T_KT_K^*$ is strictly positive if and only if
the operator $T_R +  W_{obs}(P_2-P_1)W_{obs}^*$ is strictly positive.
\end{lem}

\begin{proof}[\bf Proof]
We  first recall  some elementary facts concerning Hankel operators. To this end, let
\[
H_G = \begin{bmatrix}
        G_1 & G_2 & G_3 & \cdots \\
        G_2 & G_3 & G_4 & \cdots \\
        G_4 & G_5 & G_6 & \cdots\\
        \vdots  & \vdots  & \vdots  & \vdots  \\
      \end{bmatrix}:\ell_+^2(\mathbb{C}^p) \rightarrow \ell_+^2(\mathbb{C}^m)
\]
be the Hankel operator determined by the Taylor series
$G(z) = \sum_{\nu=0}^\infty z^\nu G_\nu$.
In a similar way, let $H_K$ be the corresponding Hankel operator
mapping  $\ell_+^2(\mathbb{C}^q)$ into $\ell_+^2(\mathbb{C}^m)$
determined by $K$. Let $W_{con,1}$ mapping
$\ell_+^2(\mathbb{C}^p)$ into $\mathbb{C}^n$  and
$W_{con,2}$ mapping
$\ell_+^2(\mathbb{C}^q)$ into $\mathbb{C}^n$
be the controllability
operators defined by
\[
W_{con,j} = \begin{bmatrix}
              B_j & A B_j & A^2 B_j  & A^3 B_j & \cdots\,\, \\
            \end{bmatrix}, \quad j=1,2.
\]
From \eqref{p1p2} we see that  $P_j = W_{con,j}W_{con,j}^*$ for $j =1,2$. Using $G_\nu = C A^{\nu -1} B_1$ for all integers $\nu \geq 1$ and the corresponding result for $K$,  we see that $H_G = W_{obs} W_{con,1}$ and
$H_K = W_{obs} W_{con,1}$. Finally,
\begin{equation}\label{hank}
H_G H_G^* =  W_{obs} P_1  W_{obs}^*
\quad \mbox{and}\quad
H_K H_K^* =  W_{obs} P_2  W_{obs}^*.
\end{equation}

Next, notice the Toeplitz operators $T_{GG^*}$ and $T_{KK^*}$ are given by the following identities:
\[
T_{GG^*} = T_GT_G^* + H_GH_G^*\ands T_{KK^*} = T_KT_K^* + H_K H_K^*.
\]
Using $R = GG^* - KK^*$, we have $T_R=T_GT_G^*-T_KT_K^*+H_GH_G^*-H_KH_K^* $. But then  \eqref{hank}  yields \eqref{trp1p2}.
\end{proof}

\begin{proof}[{\bf Proof of  Theorem \ref{thmpos}}] Assume the operator $T_GT_G^*-T_KT_K^*$ is strictly positive. Then  Lemma~\ref{lempos1} tells us $T_R$ is strictly positive, and hence, see Remark \ref{remQ}, item (i) in Theorem \ref{thmpos} is fulfilled. Furthermore, applying Lemma \ref{leminv} below with
\begin{equation}\label{MTWN}
M = T_GT_G^*-T_KT_K^*, \quad T = T_R, \quad   W = W_{obs}
\ands N = P_2 - P_1,
\end{equation}
noting that $M=T+WNW^*$ is strictly positive, by the identity \eqref{trp1p2}, we see that the matrix $Q^{-1}+P_2-P_1$ is strictly positive, and  hence item (ii) in Theorem \ref{thmpos} is fulfilled. Furthermore, again in view of \eqref{trp1p2}, in this case the inversion formula \eqref{invM}  yields the formula to compute the inverse of $T_GT_G^*-T_KT_K^*$ in \eqref{defOm}.

Conversely, assume items (i) and (ii) in Theorem \ref{thmpos} are satisfied. Then item (i) implies that $T_R$ is strictly positive, as explained in Remark \ref{remQ}, and $Q=W_{obs}^*T_R^{-1}W_{obs}=W^*T^{-1}W$, using the notation of \eqref{MTWN} in the last identity. Note that item (ii) states that $Q^{-1}-N=Q^{-1}-P_1+P_2$ is strictly positive. Hence again using Lemma \ref{leminv} below and the identity \eqref{trp1p2}, we see that item (ii) implies that $T_GT_G^*-T_KT_K^*$ is strictly positive.
\end{proof}

\begin{lem}\label{leminv}
Let $M$ be an operator acting on a  Hilbert space $\sH$  such that
\begin{equation}\label{M}
M = T + W N W^*,
\end{equation}
where $T$ on $\sH$ is a strictly positive operator,  $N$ is a self adjoint operator on a Hilbert space $\sX$, and $W$ is an operator mapping $\sX$ into $\sH$
which is one-to-one and has closed range.  Set $Q = W^* T^{-1} W$. Then $Q$ is invertible. Furthermore, $M$ is strictly positive if and only if $Q^{-1} + N$ is strictly positive. Moreover, in that  case,
\begin{equation}\label{invM}
M^{-1} = T^{-1} - T^{-1} W N \big(I  + Q N )^{-1}W^* T^{-1}.
\end{equation}
\end{lem}

\begin{proof}[\bf Proof]
Replacing $M$ by $T^{-1/2}MT^{-1/2}$ and $W$ by $T^{-1/2}W$, we see that without loss of generality we may assume that $T$ is the identity operator on $\sH$. Therefore, in what follows $M = I + W N W^*$. Note that in this case $Q=W^*W$.

The fact that  $W$ is one-to-one and  has closed range, implies that $Q=W^*W$ is invertible. It follows that the Moore-Penrose left inverse $W^+$  of $W$ is well-defined and is given by  $W^+=(W^*W)^{-1}W^*=Q^{-1}W^*$. Furthermore, the orthogonal project $\BP$ on $\sH$ mapping $\sH$ onto the range of $W$ is given $\BP=WQ^{-1}W^*$. Now note that
\begin{equation}\label{prM}
M = I + W N W^*=I-\BP  +\BP+ W N W^*=I-\BP+W(Q^{-1}+N)W^*.
\end{equation}
Put $\sH_1=\im W$ and  $\sH_0=\kr W^*$, and consider the operators
\begin{align*}
    &\t_1: \sH_1\to \sH, \quad  \t_1 u=u \quad (u\in \sH_1),  \\
    &W_1: \sX\to \sH_1, \quad  W_1 x=Wx \quad (x\in \sX).
\end{align*}
 Note that $W=\t_1W_1$ and $W^*=W_1^*\t_1^*$. Furthermore, $W_1$ is invertible, and $W_1^{-1}=Q^{-1}W^* \t_1$ . Using  \eqref{prM} we see that relative to the orthogonal decomposition $\sH=\sH_0 \oplus \sH_1$  the operator $M$ admits the following block operator matrix representation:
 \[
 M=\begin{bmatrix} I_{\sH_0}&0\\ 0 & W_1(Q^{-1}+N)W_1^*\end{bmatrix}
 \]
Since   $W_1$ is invertible, it follows that $M$ is strictly positive if and only $Q^{-1}+N$ is strictly positive. Moreover, in that case
\begin{align*}
M^{-1}&= \begin{bmatrix} I_{\sH_0}&0\\ 0 & W_1^{-*}(Q^{-1}+N)^{-1}W_1^{-1}\end{bmatrix}   \\
    &  =\begin{bmatrix} I_{\sH_0}&0\\ 0 & \t_1^* WQ^{-1}(Q^{-1}+N)^{-1}Q^{-1}W^* \t_1\end{bmatrix}
\end{align*}
It follows that
\begin{align*}
M^{-1}&= I-\BP+  WQ^{-1}(Q^{-1}+N)^{-1}Q^{-1}W^*\\
&=I- WQ^{-1}W^*+WQ^{-1}(Q^{-1}+N)^{-1}Q^{-1}W^*\\
&=I- W\Big(Q^{-1}-Q^{-1}(Q^{-1}+N)^{-1}Q^{-1}\Big)W^*
\end{align*}
Finally, note that
\begin{align*}
    & Q^{-1}-Q^{-1}(Q^{-1}+N)^{-1}Q^{-1}= \\
    &\hspace{2cm} = Q^{-1}- (Q^{-1}+N-N) (Q^{-1}+N)^{-1}Q^{-1}\\
    &\hspace{2cm} =N(Q^{-1}+N)^{-1}Q^{-1}= N(I+QN)^{-1}.
\end{align*}
This proves   \eqref{invM}.
\end{proof}

For a version of  Lemma \ref{leminv} with $T$ just nonnegative, not necessarily strictly positive, see Lemma 2.10 in \cite{FtHK13}.

\setcounter{equation}{0}
\section{Proof of Theorem \ref{mainthm}}\label{secrealX}

In this section,  we will convert the infinite dimensional state space
realization for the central solution $X$ in \eqref{realx} to the
finite dimensional realization for $X$ in \eqref{realx0},  and in the mean time
prove  Theorem \ref{mainthm}. Throughout  $G$ and $K$ are the rational matrix functions described by the observable stable realization \eqref{realGK}, and we assume that $T_G T_G^* - T_K T_K^*$ is strictly positive. Thus items (i) and (ii) in Theorem \ref{thmpos} are satisfied.  In what follows we shall freely use  the notations introduced in these two  items. In particular, the operator $T_R$ is strictly positive,  $Q$ is the stabilizing solution of the algebraic Riccati equation \eqref{are}, and the matrix $\de= R_0 - \Gamma^* Q \Gamma$ is strictly positive. We set
\begin{equation}\label{defAC04}
C_{0}= {\de^{-1}}( C - \Gamma^* Q  A ) \ands A_0 =A-\Gamma C_{0}.
\end{equation}
By item (c) in  Theorem \ref{thmpos} the matrix $A_0$ is stable. Using $C_0$ and $A_0$ in \eqref{defAC04},  the Riccati equation \eqref{are} can be rewritten as a Stein equation:
\begin{equation} \label{Steineq4}
Q-A^*QA_0=C^*C_0.
\end{equation}
The observability operator for  the pair $\{C_{0},A_{0}\}$ is the operator $W_0$ defined by
\begin{equation}\label{defwoo}
 W_{0} = \begin{bmatrix}   C_{0}  \\
   C_{0} A_{0}  \\
   C_{0} A_{0}^2 \\
   \vdots
 \end{bmatrix}:\mathbb{C}^n\rightarrow \ell_+^2(\mathbb{C}^m).
\end{equation}
We need the following lemma (cf., identity (3.19) in \cite{FKR2a}):
\begin{lem}\label{lemW0}
The operator $W_0$ is one-to-one, $ T_R^{-1} W_{obs} = W_{0}$, and
\begin{equation}
\label{idR4}
R(z)C_0(I_n-zA_0)^{-1}=  C(I_n-zA)^{-1}+\ga^*(zI_n-A^*)^{-1}Q.
\end{equation}
\end{lem}

\begin{proof}[\bf Proof]
Let us assume that \eqref{idR4} has been proved. Note that the matrix function $C(I_n-zA)^{-1}$ is a stable rational function, while $\ga^*(zI_n-A^*)^{-1}Q$ is a rational function which is analytic on the exterior of the open unit disc and has the value zero at infinity. But then \eqref{idR4}  implies that $T_RW_0x$ is equal to  $W_{obs}x$ for each $x\in \BC^n$. Since $T_R$ is invertible, we get $W_0 = T_R^{-1} W_{obs}$. Recall that  $W_{obs}$ is one to one. Therefore $W_0 = T_R^{-1} W_{obs}$ is also one to one.

It remains to prove \eqref{idR4}. To do this we use the realization \eqref{eqR}. Using $\ga C_0=A-A_0$,  a standard calculation shows that
\[
zC(I_n-zA)^{-1}\ga C_0(I_n-zA_0)^{-1}=C(I_n-zA)^{-1} - C(I_n-zA_0)^{-1}.
\]
Analogously, using the Stein equation \eqref{Steineq4}, one computes that
\begin{align*}
&\ga^*(zI_n-A^*)^{-1}C^*C_0(I_n-zA_0)^{-1}=\ga^*(zI_n-A^*)^{-1}Q+\\
&\hspace{6cm}+\ga^*QA_0(I_n-zA_0)^{-1}.
\end{align*}
Using the realization \eqref{eqR} the two preceding identities yield
\begin{align*}
R(z)C_0(I_n-zA_0)^{-1}&= C(I_n-zA)^{-1}+\ga^*(zI_n-A^*)^{-1}Q+\\
&\hspace{2cm}+ (-C+R_0C_0+\ga^*QA_0)(I_n-zA_0)^{-1}.
\end{align*}
Next  using the two identities in \eqref{defAC04}  we see that
\begin{align*}
-C+R_0C_0+\ga^*QA_0&= -C+R_0C_0+\ga^*QA-\ga^*Q\ga C_0\\
&=-(C-\ga^*QA)+(R_0-\ga^*Q\ga) C_0\\
&= -(C-\ga^*QA)+\de C_0=0.
\end{align*}
This proves \eqref{idR4}.
\end{proof}

\begin{proof}[\bf Proof of Theorem \ref{mainthm}]
In the course of this proof we shall often use the following identity (which follows from \eqref{QTRinv} and Lemma \ref{lemW0}):
\begin{equation}\label{w0w}
W_{obs}^*W_0 =Q.
\end{equation}

\paragraph{The Schur complement for $T_R$.}
From the realization  \eqref{eqR} it follows  that $T_R$ admits a  block $2\ts2$ matrix representation
\[
T_R = \begin{bmatrix}
        R_0            &    \Gamma^* W_{obs}^*\\
        W_{obs} \Gamma & T_R \\
      \end{bmatrix}\mbox{ on } \begin{bmatrix}
        \mathbb{C}^m\\
        \ell_+^2(\mathbb{C}^m)\\
      \end{bmatrix}.
\]
Since $T_R$  is invertible, the \emph{Schur complement} with respect to the $(2,2)$ entry  is given by
\[
R_0 - \Gamma^* W_{obs}^* T_R^{-1}  W_{obs} \Gamma = R_0 - \Gamma^* Q \Gamma=\de.
\]
It follows (see. e.g., \cite[page 29]{BGKR08}) that the inverse of $T_R$ is given by
\begin{align*}\label{schurinv}
T_R^{-1} = \begin{bmatrix}
       {\de^{-1}}                            &   -{\de^{-1}} \Gamma^* W_{obs}^*T_R^{-1}\\
        - T_R^{-1} W_{obs}\Gamma {\de^{-1}}  & T_R^{-1} + T_R^{-1}W_{obs}\Gamma{\de^{-1}} \Gamma^* W_{obs}^*T_R^{-1}\\
      \end{bmatrix}\mbox{ on } \begin{bmatrix}
        \mathbb{C}^m\\
        \ell_+^2(\mathbb{C}^m)\\
      \end{bmatrix}.
\end{align*}

\paragraph{A matrix representation for $\Xi$.}
To compute a finite dimensional  realization for our central (maximum entropy) solution $X$ in \eqref{realx}, we need a formula for $\Xi$ involving the state space data. Note that \eqref{realGK} is equivalent to the following two realizations:
\begin{equation}\label{reprGK5}
G(z) =  D_1  + z C(I_n -  z A)^{-1}B_1, \quad
K(z) =  D_2  + z C(I_n -  z A)^{-1}B_2.
\end{equation}
Using  the realization for $K$ in \eqref{reprGK5} and \eqref{defOm}, we obtain
\begin{align*}
\Xi &=  (T_GT_G^* - T_KT_K^*)^{-1}T_KE_q\\
&= T_R^{-1}T_KE_q + T_R^{-1}W_{obs}\om W_{obs}^* T_R^{-1}T_KE_q \qquad \mbox{[by \eqref{defOm}]}\\
&= T_R^{-1}T_KE_q + W_0\om W_0^* T_KE_q\qquad \mbox{[by Lemma \ref{lemW0}]}\\
&= T_R^{-1}T_KE_q +  W_{0}\om \begin{bmatrix}
                                     C_0^*  & A_0^* W_0^* \\
                                   \end{bmatrix}\begin{bmatrix}
                                     D_2  \\ W_{obs}B_2  \\
                                   \end{bmatrix}\\
&=T_R^{-1}T_KE_q +  W_{0}\om \left( C_0^* D_2 + A_0^* Q B_2\right) \qquad \mbox{[by \eqref{w0w}]}\\
&=T_R^{-1}T_KE_q +  W_{0}\om C_2^* \qquad \mbox{[with $C_2$ as in \eqref{defC12}]}.
\end{align*}
To compute $T_R^{-1}T_KE_q$, we use the $2\ts 2$ operator matrix representation of $T_R^{-1}$ given above. This yields
\begin{align*}
T_R^{-1} T_K E_q &=\begin{bmatrix}
       {\de^{-1}}                            &   -{\de^{-1}} \Gamma^* W_{obs}^*T_R^{-1}\\
        - T_R^{-1} W_{obs}\Gamma {\de^{-1}}  & T_R^{-1} + T_R^{-1}W_{obs}\Gamma{\de^{-1}} \Gamma^* W_{obs}^*T_R^{-1}\\
      \end{bmatrix} \begin{bmatrix}
                      D_2 \\
                      W_{obs}B_2 \\
                    \end{bmatrix}\\
&= \begin{bmatrix}
{\de^{-1}} D_2 - {\de^{-1}}\Gamma^*Q B_2\\
W_0 \left(B_2   + \Gamma{\de^{-1}} \Gamma^* Q B_2- \Gamma{\de^{-1}} D_2 \right) \\
\end{bmatrix}.
\end{align*}
Substituting this into our previous formula for $\Xi$ we arrive at
\begin{align}
\Xi &= \begin{bmatrix}
         {\de^{-1}} D_2 - {\de^{-1}}\Gamma^*Q B_2 + C_0\om C_2^*\\
         W_0\left(B_2   + \Gamma{\de^{-1}} \Gamma^* Q B_2- \Gamma{\de^{-1}} D_2 +  A_0 \om C_2^*\right) \\
       \end{bmatrix}
       =\begin{bmatrix}
         {D_0}\\
         W_0B_0 \\
       \end{bmatrix}.
       \label{Xi1}
%     & =\begin{bmatrix}
%         {D_0}\\
%         W_0B_0 \\
%       \end{bmatrix}. \label{Xi1}
\end{align}
See Theorem \ref{mainthm}  for the definitions of ${D_0}$ and $B_0$.
Finally,  note that the identity \eqref{Xi1} also shows that $S_m^* \Xi = W_0 B_0$.\smallskip

\paragraph{The state space operator $F$.}
The state space operator
\begin{equation}\label{F}
F = S_m^* - S_m^* \Xi D_V^{-1} E_q^* T_K^* = S_m^* - W_0 B_0 D_V^{-1} E_q^* T_K^*.
\end{equation}
The state space realizations for $G$ and $K$ in \eqref{reprGK5} yield
\begin{align}\label{c1c2}
E_p^* T_G^* W_0 &= \begin{bmatrix}
                   D_1^* & B_1^* W_{obs}^* \\
                 \end{bmatrix}\begin{bmatrix}
                   C_0 \\ W_0 A_0\\
                 \end{bmatrix} = D_1^* C_0 + B_1^* Q A_0 = C_1,\nonumber\\
E_p^* T_K^* W_0 &= \begin{bmatrix}
                   D_2^* & B_2^* W_{obs}^* \\
                 \end{bmatrix}\begin{bmatrix}
                   C_0 \\ W_0 A_0\\
                 \end{bmatrix} = D_2^* C_0 + B_2^* Q A_0 = C_2.
\end{align}
Using the definition of $C_2$ with $S_m^* W_0 = W_0 A_0$,  we obtain
\begin{align*}
F W_0 &= S_m^*W_0 - W_0 B_0 D_V^{-1} E_q T_K^* W_0
=W_0 A_0 - W_0 B_0 D_V^{-1} C_2.
\end{align*}
This leads to the following intertwining relation:
\begin{equation}\label{Fw}
F W_0 = W_0 A^\times
\quad \mbox{where}\quad A^\times = A_0 -  B_0 D_V^{-1} C_2.
\end{equation}
This readily implies that
\begin{equation}\label{Fw1}
(I - z F)^{-1} S_m^* \Xi = (I - z F)^{-1} W_0 B_0=
W_0 (I_n - z A^\times)^{-1}B_0.
\end{equation}
Substituting the previous formulas into our state
space formula for $X$ in \eqref{realx}, we obtain
\begin{align}\label{realxF}
X(z) &= D_U D_V^{-1} +
 z\left(E_p^* T_G^* - D_U D_V^{-1} E_q^* T_K^* \right)  (I - z F)^{-1} S_m^* \Xi  D_V^{-1}\nonumber\\
%  &= D_U D_V^{-1} +
% z\left(E_p^* T_G^* - D_U D_V^{-1} E_q^* T_K^* \right)  (I - z F)^{-1}W_0 B_0D_V^{-1}\nonumber\\
  &= D_U D_V^{-1} +
 z\left(E_p^* T_G^* - D_U D_V^{-1} E_q^* T_K^* \right)W_0  (I_n - z A^\times)^{-1}B_0 D_V^{-1}\nonumber\\
 &= D_U D_V^{-1} +
 z\left(C_1 - D_U D_V^{-1}C_2\right) (I_n - z A^\times)^{-1}B_0 D_V^{-1}.
 \end{align}

\paragraph{Computing $D_U$ and $D_V$.}
To complete our finite dimensional state space realization formula for $X$ in \eqref{realx0}, we need an
expression for $D_U$ and $D_V$, that is,
\begin{align*}
D_U &= E_p^* T_G^*  \Xi = \begin{bmatrix}
                           D_1^* & B_1^* W_{obs}^*\\
                         \end{bmatrix} \begin{bmatrix}
                           {D_0} \\ W_0 B_0\\
                         \end{bmatrix} = D_1^* {D_0} + B_1^* Q B_0\\
 D_V &= I_q + E_q^* T_K^*  \Xi = I_q  +  \begin{bmatrix}
                           D_2^* & B_2^* W_{obs}^*\\
                         \end{bmatrix} \begin{bmatrix}
                           {D_0} \\ W_0 B_0\\
                         \end{bmatrix} = I+ D_2^* {D_0} + B_2^* Q B_0.
\end{align*}
This together with \eqref{realxF} yields the finite dimensional state space formula for the maximum entropy solution $X$ in \eqref{realx0}.

It is noted that $V(0) = D_V$. So  the entropy $\mathcal{E}(X) = - \ln \det[D_V]$;
see \eqref{entclt}.\smallskip

\paragraph{The function $V$ is invertible outer and $A^\ts$ is stable.}
Recall that $S_m^* \Xi = W_0 B_0$. Using this with $S_m^* W_0 = W_0 A_0$  and
the state space realization for $V$ in \eqref{realV}, we obtain
\begin{align*}
V(z) &=  D_V + z E_q^* T_K^*  (I - z S_m^*)^{-1}  S_m^* \Xi\\
&=D_V + z E_q^* T_K^*  (I - z S_m^*)^{-1}  W_0 B_0\\
 &=D_V + z E_q^* T_K^*W_0  (I_n - z A_0)^{-1} B_0.
\end{align*}
Since  $C_2 = E_q^* T_K^* W_0$,  a finite dimensional realization for
$V$ is given by
\begin{equation}\label{realV5}
V(z) = D_V + z C_2  (I_n - z A_0)^{-1} B_0.
\end{equation}
Since $A_0$ is stable, the function $V$ is a stable rational matrix function. In particular, $V$ belongs to $H_{q\ts q}^\iy $. From Theorem \ref{thcl} we know that $V^{-1}$ belongs to $H_{q\ts q}^\iy $. Thus both $V$ and $V^{-1}$ are in  $H_{q\ts q}^\iy $, and so $V$ is invertible outer.

Next we prove that $A^\ts$ is stable.  By employing a standard state space inversion formula, the  inverse for $V$ given by
\[
V(z)^{-1} = D_V^{-1} - z D_V^{-1} C_2 (I_n - z  A^\times)^{-1} B_0D_V^{-1}
\]
where $A^\times = A_0 - B_0 D_V^{-1} C_2$. We already  know that $V^{-1}$ belongs to $H_{q\ts q}^\iy $.  Hence   $V^{-1}$ is also a stable rational matrix function. Because $A_0$ is stable and $V(z)^{-1}$ is analytic in the closed unit disc, Theorem 2.1  in \cite{BGKR08}  tells us that $A^\ts$ is stable.\smallskip

\paragraph{The  solution $X$ is strictly contractive.} It remains to show that $\|X\|_\iy<1$. From Theorem \ref{thcl} we know that $\tht=V(0)V^{-1}$ is the outer spectral factor of $I-X^*X$. However, as proved in the preceding paragraph, the function  $V^{-1}$ is invertible outer. Hence  $\tht$ is invertible outer.  The latter implies that $I-X^*(\z)X(\z)$ is strictly positive for each $\z\in \BT$. Therefore $\|X(\z)\|<1$ for $\z\in \BT$.  Thus  $X$ is a strictly contractive solution to our Leech problem.
\end{proof}

%:new remark 5.2
\begin{remark}For later purposes (see the next remark) we mention that $U$ is given by the following finite dimensional realization:
\begin{equation}\label{realU5}
U(z) = D_U + z C_1  (I_n - z A_0)^{-1} B_0.
\end{equation}
The proof is similar to the proof of the realization of $V$ in \eqref{realV5}. Indeed, using    $S_m^* W_0 = W_0 A_0$  with
the state space realization for $U$ in \eqref{realU}, we obtain
\begin{align*}
U(z) &=  D_U + z E_p^* T_G^*  (I - z S_m^*)^{-1}  S_m^* \Xi\\
&=D_U + z E_p^* T_G^*  (I - z S_m^*)^{-1}  W_0 B_0\\
 &=D_U + z E_p^* T_G^*W_0  (I_n - z A_0)^{-1} B_0.
\end{align*}
\end{remark}

%:new remark 5.3
\begin{remark}\label{rem53} Given the various matrices appearing in Theorems~\ref{thmpos} and \ref{mainthm} one can now also prove directly that the function $X=UV^{-1}$ given by \eqref{realx0} satisfies $GX=K$, independent of the operator theory result based on the commutant lifting theorem. To illustrate this we give a direct proof of the identity  $GU=KV$, using the realizations of $U$ and $V$ given by \eqref{realU5} and \eqref{realV5}, respectively.  The direct proof requires a number of non-trivial identities which are given by the following lemma.
\end{remark}

\begin{lem}\label{lem-id}
Let $\begin{bmatrix}  G & K \\  \end{bmatrix}$ be given by  the observable  stable realization in \eqref{realGK}, and assume that items $(i)$ and $(ii)$ in \textup{Theorem \ref{thmpos}} are satisfied. Define $\om_0=I+(P_2- P_1)Q $.  Then the following identities hold:
\begin{align}
&B_1C_1  - B_2C_2 =A\om_0 - \om_0A_{0} \label{BC1},\\
& D_1C_1  -  D_2C_2   = C\om_0, \label{DC1}\\
&B_1 D_U - B_2 D_V  =  -\om_0 B_0, \label{b1dub2dv}\\
&D_1 D_U - D_2 D_V =   0.  \label{d1dud2dv}
\end{align}
Here $C_1$ and $C_2$ are given by \eqref{defC12}, and the matrices $A_{0}$ and $C_{0}$ are  given by \eqref{defAC04}.
\end{lem}

For the moment let  us assume that the above  identities are proved, and let us consider
$G(z)U(z)-K(z)V(z)$. Using the realizations in  \eqref{reprGK5}, \eqref{realU5}, and \eqref{realV5}, we see that
\begin{align*}
&G(z)U(z)-K(z)V(z)=(D_1 D_U-D_2 D_V) +\\
&\hspace{2cm}+zC(I_n-zA)^{-1}(B_1 D_U-B_2 D_V) \\
&\hspace{2cm}+z(D_1 C_1-D_2 C_2)(I_n-zA_0)^{-1}B_0\\
&\hspace{2cm}+zC(I_n-zA)^{-1}(zB_1 C_1- zB_2 C_2)(I_n-zA_0)^{-1}B_0.
\end{align*}
Now using the identity \eqref{BC1} we see that
\begin{align*}
& (I_n-zA)^{-1}(zB_1 C_1- zB_2 C_2)(I_n-zA_0)^{-1}=\\
&\hspace{1.5cm}=(I_n-zA)^{-1}(zA\om_0-z\om_0A_0)(I_n-zA_0)^{-1}\\
&\hspace{1.5cm}=(I_n-zA)^{-1}\big(\om_0(I_n-zA_0)-(I_n-zA)\om_0\big)(I_n-zA_0)^{-1}\\
&\hspace{1.5cm}= (I_n-zA)^{-1}\om_0-\om_0(I_n-zA_0)^{-1}.
\end{align*}
It follows that
\begin{align*}
&G(z)U(z)-K(z)V(z)=(D_1 D_U-D_2 D_V) +\\
&\hspace{2cm}+zC(I_n-zA)^{-1}(B_1 D_U-B_2 D_V+\om_0 B_0) \\
&\hspace{2cm}+z(D_1 C_1-D_2 C_2-C\om_0)(I_n-zA_0)^{-1}B_0.
\end{align*}
The identities \eqref{d1dud2dv}, \eqref{b1dub2dv}, and \eqref{DC1} then show that $G(z)U(z)-K(z)V(z)$ is identically equal to zero, that is, $GU=KV$.

\begin{proof}[\textbf{Proof of Lemma \ref{lem-id}}]
In the sequel we shall  use  the following two identities
\begin{equation}\label{stein2}
Q-A^*QA=C_0^*\de C_0 \ands Q -A^*Q A_{0}=C^*C_{0}.
\end{equation}
These identities  follow  by using  the definition of $A_{0}$ and $C_{0}$ together with the fact that $Q$ is a hermitian matrix satisfying \eqref{are}.

\paragraph{Proof of \eqref{BC1}.}
Using  $C_{j}= D_j^* C_0 + B_j^*Q A_0$ for $j = 1,2$ and
the second Stein equation in~\eqref{stein2}, we have
\begin{align*}
B_jC_{j}&=B_j D^*C_{0}+B_jB_j^*Q A_{0}=B_jD_j^*C_{0}+(P_j - AP_jA^*)Q A_{0}\quad [\mbox{by \eqref{p1p2}}]\\
&= B_j D_j^*C_{0}+P_jQ A_{0}-AP_jA^* Q A_{0}\\
&= B_jD_j^*C_{0}+P_jQ A_{0}-AP_j(Q -C^*C_{0})\quad \mbox{[by the second part of  \eqref{stein2}]}\\
&= (B_jD_j^*+AP_jC^*)C_{0}+P_jQ A_{0}-AP_jQ \quad (j=1,2).
\end{align*}
Taking differences we obtain:
\begin{align*}
B_1C_{1} - B_2C_{2} &=
\big (B_1D_1^*- B_2D_2^* +A(P_1- P_2)C^*\big)C_{0}+\\
 &\hspace{4cm} +(P_1- P_2)Q A_{0}-A(P_1 - P_2)Q\\
&= \ga C_{0}+(P_1- P_2)Q A_{0}-A(P_1 - P_2)Q\qquad \mbox{[by \eqref{defGa}]}\\
& =A-A_{0} +(P_1- P_2)Q A_{0}-A(P_1 - P_2)Q \qquad \mbox{[by \eqref{defAC04}]}\\
&=A\big(I  +(P_2- P_1) Q \big) - \big(I +(P_2- P_1)Q \big)A_{0}=A\Omega_0-\Omega_0A_0.
\end{align*}
Hence \eqref{BC1} holds.

\paragraph{Proof of \eqref{DC1}.} Again using  $C_{j}= D_j^* C_0 + B_j^*Q A_0$ for $j = 1,2$  we have
\begin{align*}
D_1 C_{1}- D_2 C_2&= \left(D_1D_1^* - D_2D_2^*\right)C_{0}   +  \left(D_1 B_1^* - D_2 B_2^*\right)Q A_{0}\\
&= \left(D_1D_1^* - D_2D_2^*\right)C_{0} +(\ga^* -C(P_1 - P_2) A^*)Q A_{0}\quad \mbox{[by \eqref{defGa}]}\\
&= \left(D_1D_1^* - D_2D_2^*\right) C_{0} +   \ga^*Q A_{0}-C(P_1 - P_2)A^*Q A_{0}\\
&=\left(D_1D_1^* - D_2D_2^*\right) C_{0} +\ga^*Q (A-\ga C_{0})+  \qquad \mbox{[by \eqref{defAC04}]}\\
&\hspace{3cm}-C(P_1 - P_2)(Q -C^*C_{0})  \qquad \mbox{[by \eqref{stein2}]}\\
&=\big(D_1D_1^* - D_2D_2^* +C(P_1 - P_2)C^*\big)C_{0}+\\
&\hspace{3cm} + \ga^*Q A-\ga^*Q \ga C_{0}-C(P_1 - P_2)Q \\
&=(R_0-\ga^*Q \ga)C_{0} +\ga^*Q A-C(P_1 - P_2)Q\qquad   \mbox{[by  \eqref{defR0}]}\\
&= C-\ga^*Q A+\ga^*Q A-C(P_1 - P_2)Q   \qquad  \mbox{[by
\eqref{defAC04}}]\\
&= C\big( I +(P_2- P_1) Q \big)=C\Omega_0.
\end{align*}
Thus \eqref{DC1} holds.

\paragraph{Proof of \eqref{b1dub2dv}.}
To establish \eqref{b1dub2dv} we use that $B_0 = B_2 - \Gamma D_0 + A \om C_2^*$.  This identity follows from
\begin{align}
B_0 &= B_2   - \ga\de^{-1}(D_2 - \ga^*Q B_2) +  A_0 \om C_2^* \nonumber\\
&= B_2   - \ga\de^{-1}(D_2 - \ga^*Q B_2) +  (A - \Gamma C_0)  \om C_2^*\nonumber\\
&= B_2   - \ga D_0 +  A  \om C_2^*. \label{b00}
\end{align}
Using  \eqref{defGa} we see that
\begin{align*}
&B_1 D_U - B_2 D_V +\Omega_0 B_0=\\
&\hspace{.5cm}= - B_2 + (B_1 D_1^* - B_2 D_2^*) D_0 +
(B_1 B_1^* - B_2 B_2^*) Q B_0+\Omega_0 B_0\\
&\hspace{.5cm}= - B_2 + \big(\Gamma + A(P_2 - P_1)C^*\big) D_0 + (B_1 B_1^* - B_2 B_2^*) Q B_0+\Omega_0 B_0\\
&\hspace{.5cm}= - B_2+\ga D_0+ A(P_2 - P_1)C^*D_0 + (B_1 B_1^* - B_2 B_2^*) Q B_0+\\
&\hspace{2cm}+\big(I+(P_2- P_1)Q\big)B_0\\
&\hspace{.5cm}=A  \om C_2^*   + A(P_2 - P_1)C^*D_0 + (B_1 B_1^* - B_2 B_2^*) Q B_0 +(P_2- P_1)QB_0,
\end{align*}
using \eqref{b00} in the last identity. Next we use  \eqref{p1p2}. This yields
\begin{align*}
&B_1 D_U - B_2 D_V +\big(I+(P_2- P_1)Q \big)B_0=\\
&\hspace{.5cm}= A  \om C_2^* +A(P_2 - P_1)C^*D_0 +\\
&\hspace{2cm}+ \big((P_1-P_2)-A(P_1-P_2)A^*\big)QB_0+(P_2- P_1)QB_0 \\
&\hspace{.5cm}= A \om C_2^* +  A(P_2 - P_1)(C^*D_0+A^*QB_0).
\end{align*}
We proceed by computing  $C^*D_0+A^*QB_0$. We have
\begin{align*}
C^*D_0+A^*QB_0&=C^*D_0+A^*QB_2-A^*Q\ga D_0+A^*QA\om C_2^* \quad [\mbox{by \eqref{b00}}]\\
&=(C^*-A^*Q\ga )D_0+A^*QB_2+A^*QA\om C_2^*\\
&=C_0^*\de D_0+A^*QB_2+A^*QA\om C_2^*.
\end{align*}
Now observe that $C_0^*\de D_0=C_0^*(D_2-\ga^*QB_2)+C_0^*\de C_0\om C_2^*$. Using this together with the first identity in \eqref{stein2} we obtain
\begin{align*}
C^*D_0+A^*QB_0&= C_0^*D_2 +(A^*-C_0^*\ga^*)QB_2+A^*QA\om C_2^*+C_0^*\de C_0\om C_2^*\\
&=C_0^*D_2 + A_0^*Q B_2+Q\om C_2^*\\
&=C_2^*+ Q\om C_2^*=(I_n+Q\om)C_2^*.
\end{align*}
Summarizing we have:
\[
B_1 D_U - B_2 D_V +\Omega_0 B_0=A \Big(\om   +  (P_2 - P_1)+ (P_2 - P_1)Q \om \Big)C_2^*.
\]
Now write $\om$ in \eqref{defOm} as $\om=-N(I+QN)^{-1}$, where $N=P_2-P_1$. We see that
\begin{align}
&\om  +  (P_2 - P_1)+ (P_2 - P_1)Q \om=\nonumber\\
&\hspace{1cm}=-N(I+QN)^{-1}+N+NQ\Big(-N(I+QN)^{-1}\Big)\nonumber\\
&\hspace{1cm}=-N(I+QN)^{-1}+N -N(QN+I-I)(I+QN)^{-1}\nonumber\\
 &\hspace{1cm}=-N(I+QN)^{-1}+N(I+QN)^{-1}=0.\label{fundeq05}
 \end{align}
Therefore we obtain \eqref{b1dub2dv}.

\paragraph{Proof of \eqref{d1dud2dv}.} To establish \eqref{d1dud2dv} notice that
\begin{align*}
&D_1 D_U - D_2 D_V =
D_1 \left(D_1^* {D_0} + B_1^* Q B_0\right) - D_2 - D_2\left(D_2^* {D_0} + B_2^* Q B_0\right)\\
&\hspace{.5cm}=\left(D_1D_1^* - D_2D_2^*\right)D_0 +\left(D_1B_1^* - D_2B_2^*\right)Q B_0 - D_2\\
&\hspace{.5cm}= \left(R_0 + C(P_2-P_1)C^*\right)D_0 - D_2+\,\hspace{3.8cm}  \mbox{[use  \eqref{defR0}]}\\
&\hspace{1cm}+ \left(\Gamma^* + C(P_2-P_1) A^*\right)(QB_2 - Q\Gamma D_0 + Q A \om C_2^*)\hspace{.4cm} \mbox{[use  \eqref{defGa} and \eqref{b00}]}\\
&\hspace{.5cm}=\Big(R_0 - \Gamma^* Q \Gamma+ C(P_2-P_1)(C^*- A^* Q \Gamma)\Big)D_0 +\\
&\hspace{1cm} + \Big(\Gamma^* Q + C(P_2-P_1) A^*Q\Big)\Big(B_2 +  A \om C_2^*\Big) - D_2\\
&\hspace{.5cm}= \Big(\Delta + C(P_2-P_1)C_0^*\Delta\Big)\Big( \Delta^{-1}(D_2 - \Gamma^*Q B_2) + C_0\om C_2^*\Big) +\\
&\hspace{1cm} + \Big(\Gamma^* Q + C(P_2-P_1) A^*Q\Big)\Big(B_2 +  A \om C_2^*\Big) - D_2\\
&\hspace{.5cm}= C(P_2-P_1)C_0^*D_2+
C(P_2-P_1)\Big(A^* - C_0^* \Gamma^*\Big)QB_2+\\
&\hspace{1cm} + \Big( \Delta C_0  + C(P_2-P_1)C_0^*\Delta C_0+ \Gamma^* QA + C(P_2-P_1) A^*QA\Big)\om C_2^*  \\
&\hspace{.5cm}=C(P_2-P_1)C_2^* +\\
&\hspace{1cm}+ \Big( \Delta C_0+ C(P_2-P_1)Q+ \Gamma^* QA \Big)\om C_2^* \quad \mbox{[by the first part of \eqref{stein2}]}\\
&\hspace{.5cm}= C \Big(P_2-P_1 + \om + (P_2-P_1)Q \om \Big) C_2^*=0,\  \mbox{[because of \eqref{fundeq05}].}
\end{align*}
Therefore \eqref{d1dud2dv} holds.
\end{proof}

As the identities in Lemma \ref{lem-id} show the matrices appearing in Theorems~\ref{thmpos} and \ref{mainthm}  have a lot of structure. As a further illustration of this fact we mention without proof  the following identity:
\[
C_{1}^*C_1-C_2^*C_{2} =  \Big( Q + Q   (P_2 - P_1) Q\Big)  - A_{0}^*\Big(Q  + Q   (P_2 - P_1) Q  \Big)A_{0}.
\]
See also \eqref{idR4}.
%:ref

\end{document}